\documentclass[a4paper]{amsart}
\usepackage[utf8]{inputenc}
\usepackage{xcolor}
\usepackage{ulem}
\usepackage{latexsym}
\input amssym.def
\input amssym.tex

\newtheorem{definition}{Definition}[section]

\newtheorem{theorem}[definition]{Theorem}

\newtheorem{remark}[definition]{Remark}
\newtheorem{example}[definition]{Example}

\def\emph#1{{\bfseries\itshape{#1}}}
 1   
 1  

\def\R{\mathbb{R}}               
\def\S{\mathbb{S}}               

\def\lcf{\lbrack\! \lbrack}
\def\rcf{\rbrack\! \rbrack}
\def\so{\mathfrak{so}}
\def\g{\mathfrak{g}}

\newcount\ancho
\newcount\anchom
\newcount\anchoa
\newcount\anchob
\newcommand{\ltilde}[2]{\ancho=#1 \anchom=\ancho \divide\anchom by 2
            \anchoa=\ancho \divide\anchoa by 4
        \anchob=\anchom \advance\anchob by \anchoa
      $\kern-5pt \begin{array}[b]{c}
                 \begin{picture}(1,1)(\anchom,0)
         \qbezier(0,2)(\anchoa,5)(\anchom,2)
         \qbezier(\anchom,2)(\anchob,-1)(\ancho,4)
         \qbezier(0,2)(\anchoa,4.5)(\anchom,1.8)
         \qbezier(\anchom,1.8)(\anchob,-1.5)(\ancho,4)
      \end{picture} \\[-4pt]
       \mbox{#2}
       \end{array} \kern-9pt$
       }

\begin{document}

\title[Jacobi Reeb dynamics]{Mechanical Hamiltonian systems with respect to linear Poisson structures and Jacobi-Reeb dynamics}
\author{D. Iglesias Ponte, J. C. Marrero$^*$, E. Padr\'on}

\date{}
\thanks{$^*$ Corresponding author}

\begin{abstract}
\noindent In this paper, we present a relation between Jacobi-Reeb dynamics and the dynamics associated with a mechanical Hamiltonian system with respect to a linear Poisson structure on a vector bundle. For this purpose, we will use the so-called Jacobi bundle metrics induced by the mechanical Hamiltonian system. These constructions extend classical results on the relation between standard mechanical Hamiltonian systems on cotangent bundles and Reeb dynamics.
\end{abstract}

\maketitle
\begin{center}
{\small\it ULL-CSIC Geometr\'{\i}a Diferencial y Mec\'anica Geom\'etrica, Departamento de Matem\'aticas, Estad\'{\i}stica e Investigaci\'on Operativa
and Instituto de Matem\'aticas y Aplicaciones (IMAULL)}\\{\small\it University of La Laguna, Spain}
\\[5pt]
{\small\it e-mail: diglesia@ull.edu.es, jcmarrer@ull.edu.es, mepadron@ull.edu.es}
\end{center}

\setcounter{section}{0}

\section{Introduction}
\setcounter{equation}{0}
\subsection{Jacobi metrics, contact structures and Reeb dynamics}
For a mechanical system on a configuration manifold $Q$, with Hamiltonian $H_{(g,V)}\colon T^*Q\to \mathbb{R}$ given as the sum of the kinetic energy associated to a Riemannian metric $g$ on $Q$ and a potential energy $V\colon Q\to \mathbb{R}$, the Jacobi-Maupertuis Principle of Least Action states that fixing a constant energy $e$, for the points in the open subset $U_e=\{q\in Q \,\colon\, V(q)<e\}$ of $Q$, the solutions can be restricted to the energy surface $H=e$ and they minimize the action given by $\int \sqrt{2(e - V(q))} ds$ \cite{Arnold}. An interesting connection between the classical Jacobi-Maupertuis principle and the Routh reduction has been discussed in \cite{B,G}.

\medskip
From a geometric perspective, using 
the Riemannian metric $g$ and the energy 
$V$,  fixing $e\in \mathbb{R}$, there is a Riemannian 
metric $g_e$ on $U_e$, the Jacobi metric, in such a way that the 
geodesics of $g_e$ with velocity 1 are, up to 
reparametrization, the trajectories of the system with fixed energy $e$. In addition, one can construct a contact structure on the spheric cotangent bundle $\mathbb{S}_{g_e}(T^*U_e)$ of radius 1, with respect to the Jacobi metric $g_e$, and the corresponding Reeb vector field is just the restriction to $\mathbb{S}_{g_e}(T^*U_e)$ of the Hamiltonian vector field of the kinetic energy ${\kappa _{g_e}}$  associated with the Jacobi metric $g_e$ on $U_e$, that is, the geodesic flow on the Riemannian space $(U_e,g_e)$ (for more details see, for instance, \cite{Godbillon}).

Why are these results relevant? Some answers for this question are the following:

\begin{itemize}

\item In relation with geodesic flows:  The negative curvature of a compact Riemannian manifold has a strong
impact on the behavior of the geodesic flow (this is related to the exponential instability of geodesics on  the manifold). For instance, almost all phase trajectories are dense in $\kappa ^{-1}_{g_e}(\frac{1}{2})$. This is also related with the ergodicity of the system (see Anosov's monograph \cite{Anosov}; see also Arnold's book \cite{Arnold}). 

On the other hand, from the previous discussion one can deduce that the dynamics on the level set of the Hamiltonian function is essentially Reeb dynamics. This implies strong consequences. In particular, the Weinstein conjecture \cite{Weinstein} claims that on a compact contact manifold, its Reeb vector field should carry at least one periodic orbit. The Weinstein conjecture has been  proved for contact hypersurfaces in ${\mathbb {R}^{2n}}$ \cite{Viterbo},  for  cotangent bundles \cite{Hofer-Viterbo} or for  closed $3$-dimensional manifolds \cite{Taubes}. For closed higher dimension manifolds is an open problem.

\item In relation with integrability: First integrals for the mechanical Hamiltonian system $(T^*U_e,H_{(g,V)})$ with configuration space $U_e$ are just first integrals for the kinetic Hamiltonian system $(T^*U_e, \kappa_{g_e})$.
On the other hand, linear first integrals on the phase space for $g_e$ can be obtained using Noether Theorem for the cotangent lift of an isometric action for the Jacobi metric $g_e$ on $U_e$ (see, for instance, \cite{Abraham-Marsden}). But, also, quadratic (or higher order) first integrals on the phase space for $g_e$ can be obtained  looking for Killing tensors. See, for instance, \cite{Visinescu} and the references therein.
\end{itemize}
We also remark that, very recently, the classical results on the relation between standard mechanical Hamiltonian systems, geodesic flows and Reeb dynamics have been extended for standard mechanical systems subjected to linear nonholonomic constraints
(see \cite{AnMaMa}). 
 
\subsection{Poisson and Jacobi structures} As it is well-known, Poisson manifolds are a natural generalization of symplectic manifolds and play an important role in Classical Mechanics. In fact, 
Poisson brackets appear,  in a natural way, in the study of several mechanical systems such as  
systems with constraints or in the reduction of systems  with symmetries.  A particular type of 
Poisson structures is the linear one over the dual $A^*$ of a vector bundle $A\to M$, i.e. Poisson 
brackets closed for fiberwise linear functions (see \cite{Courant}). For instance, the Poisson structure induced by  the canonical symplectic $2$-form on the cotangent bundle $T^*Q$ of an $n$-dimensional manifold $Q$  
is an example of a linear Poisson structure.

\medskip
On the other hand, Jacobi manifolds are natural 
extensions of Poisson manifolds and, in addition, they 
include contact manifolds (like the case of ${\mathbb 
S}_g(T^*Q)$) or locally conformal symplectic 
manifolds. More precisely, a Jacobi structure on a 
manifold $M$ is a bracket of functions on $M$ which is a 
local Lie algebra in the sense of Kirillov \cite{kirillov}. 
Alternatively, a Jacobi structure on $M$ is a pair $(\Lambda,E)$, where $\Lambda$ is a $2$-vector and $E$ is a vector field on $M$ (the Reeb vector field of the Jacobi structure) such that $[\Lambda,\Lambda]=2\Lambda\wedge E$ and $[E,\Lambda]=0$, 
with $[\cdot ,\cdot ]$ the Schouten-Nijenhuis bracket (see \cite{Li}). For a standard mechanical Hamiltonian system $(Q,H_{(g,V)})$  and a fixed energy $e$, the spheric cotangent bundle ${\mathbb S}_{g_e}(T^*U_e)$ associated with the Jacobi metric $g_e$, admits a Jacobi structure (in fact, a contact structure) and its Reeb vector field is, up to reparametrization, the restriction to ${\mathbb S}_{g_e}(T^*U_e)$ of the Hamiltonian vector field of $H_{(g,V)}$.

\subsection{Aim of the paper}
Now, we will consider a more general situation. Let $\tau\colon A\to Q$ be a vector bundle over $Q$ such that the dual bundle $A^*$ admits a linear Poisson structure and let $H_{(g,V)}\colon A^*\to \R$ be a Hamiltonian function of mechanical type associated with a bundle metric $g$ on $A$ and a potential energy $V\colon Q\to \R\in C^\infty (Q)$. For a fixed energy $e$, we can define (as in the standard case) the corresponding Jacobi bundle metric $g_e$ on $A_e=\tau ^{-1}(U_e)$ and the spheric dual bundle ${\mathbb S}_{g_e}(A^*_e)$ associated with $g_e$. Then, a natural question arises: is it possible to define a Jacobi structure (not necessarily a contact one) on ${\mathbb S}_{g_e}(A^*_e)$ such that the Reeb vector field is the restriction to ${\mathbb S}_{g_e}(A^*_e)$ of the Hamiltonian vector field on $A^*_e$ associated with the Hamiltonian function $H_{(g,V)}$? In this paper, we will give an affirmative answer to this question (see Theorems \ref{prop5.3} and \ref{estructura-S-estrella-e}).

\subsection{Structure of the paper} The paper is organized as follows. In Section 2, we will 
review the geometric description of the relation between the dynamics of mechanical Hamiltonian systems on the cotangent bundle of a Riemannian manifold, geodesic flows and Reeb dynamics and we will present some motivating examples. In Section 3, which contains the main results of the paper, we extend the previous relation for mechanical Hamiltonian systems with respect to fiberwise linear Poisson structures on vector bundles (see Theorems \ref{prop5.3} and \ref{estructura-S-estrella-e}). In Section 4, we will apply our main results to mechanical Hamiltonian systems with respect to linear Poisson structures on the dual bundle of an Atiyah vector bundle and linear semi-direct Poisson structures. As particular cases of the previous situation, we discuss the coupled planar pendula and the heavy top. Finally, in Section 5, we point out some interesting future lines of research opened up by the results of this paper.

\section{Motivation: The classical problem}\label{subsection2.1}
First of all, we will review the relation between the dynamics of mechanical Hamiltonian systems on the cotangent bundle of a Riemannian manifold and the Reeb dynamics (for more details see, for instance, \cite{Godbillon}).

Let $(Q,g)$ be a Riemannian manifold of dimension $n$ and denote by $\lambda_Q$ the Liouville $1$-form on $T^*Q$ and by $\omega _Q=-d\lambda _Q$ the corresponding symplectic 2-form on $T^*Q$.  Consider the kinetic Hamiltonian function $\kappa _g\colon T^*Q\to \R$ induced from the metric $g$, 
given by
\[
\kappa _g(\alpha _q)=\frac{1}{2} \|\alpha _q\|_g^2,\qquad \mbox{for } \alpha _q\in T^*_qQ,
\]
where $\|\cdot,\cdot\|_g$ is the norm induced by $g$ on $T^*Q$. Then, the spheric cotangent bundle of radius 1, $\mathbb{S}_g(T^*Q)=\kappa _g^{-1}(\frac{1}{2})$, is a regular hypersurface of $T^*Q$ and
\[
\eta :=\iota ^*(\lambda _Q)
\]
defines a contact structure on $\mathbb{S}_g(T^*Q)$, being $\iota \colon \mathbb{S}_g(T^*Q)\to T^*Q$ the canonical inclusion. We recall that the  $1$-form $\eta$ is a  {\it contact $1$-form } if 
\[
\eta\wedge d\eta\wedge\cdots^{(n-1}\cdots\wedge d\eta\in \Omega^{2n-1}(\mathbb{S}_g(T^*Q))
\]
 is a volume form on $\mathbb{S}_g(T^*Q).$ On a contact manifold there is a distinguished vector field ${\mathcal R}$, {\it the Reeb vector field}, characterized by  $i_R\eta=1$ and $i_Rd\eta=0.$ In our particular case, the Hamiltonian vector field $X_{\kappa _g}\in \mathfrak{X}(T^*Q)$ of $\kappa _g$ is the geodesic flow of $g$ in $T^*Q$. This means that the trajectories of the Hamiltonian system $(Q,\kappa _g)$ (that is, the projection via $\tau ^* _Q\colon T^*Q\to Q$ of the integral curves of $X_{\kappa _g}$) are the geodesics of the metric $g$. In addition, the Reeb vector field ${\mathcal R}$ associated with the contact manifold $(\mathbb{S}_g(T^*Q),\eta)$ is just the restriction to $\mathbb{S}_g(T^*Q)$  of $X_{\kappa_g}$, that is,
\[
{\mathcal R}=(X_{\kappa _g}{})_{| \mathbb{S}_g(T^*Q)}.
\]
Indeed, using that 
$i_{X_{\kappa _g}} \lambda _Q =2\kappa _g$  and $i_{X_{\kappa _g}} \omega _Q =d \kappa _g$,
we deduce that 
\begin{eqnarray*}
i_{{\mathcal R}}  (d \iota ^* \lambda _Q )&=& - \iota ^* (i_{X_{\kappa _g}} \omega _Q) = - d (\kappa _g\circ \iota )=0,\\
i_{{\mathcal R}}   (\iota ^* \lambda _Q )&=&  \iota ^* (i_{X_{\kappa _g}} \lambda _Q) =2 (\kappa _g\circ \iota )=1.
\end{eqnarray*}
Moreover, there is a symplectomorphim between the symplectic open subset  $T^*Q- \mathcal{O}_Q$, where $\mathcal{O}_Q$ is the zero-section on $T^*Q$, and the symplectification of $(\mathbb{S}_g(T^*Q),\iota ^* (\lambda _Q))$ i.e. the manifold $\mathbb{S}_g(T^*Q)\times \R$  endowed with the symplectic structure $\omega _{\mathbb{S}_g(T^*Q)\times \R}=-e^t (d\iota ^*\lambda_Q -\iota ^*\lambda _Q\wedge dt)$. The symplectic isomorphism $\varphi  \colon  \mathbb{S}_g(T^*Q)\times \R \mapsto  T^*Q-\mathcal{O}_Q$ is given by  $(\alpha _q , t)\mapsto  \varphi  (\alpha _q, t)=e^t \, \alpha _q$.

\medskip
Now, suppose that the Hamiltonian function  is defined by
\[
H_{(g,V)}:= \kappa _g+V\circ \tau^* _Q,
\]
$V\in C^\infty (Q)$ being the potential energy.  Let $e \in \mathbb{R}$ such that
$U_e = \{ q \in Q \, |\, V(q)<e \}$
is a non-empty set of $Q$.  For instance, if $V$ is bounded above by a certain $e\in \mathbb{R}$ (which happens, in particular, when $Q$ is compact), $U_e=Q$.
Denote by $T^*U_e$ the cotangent bundle of $U_e$ given by
\[
T^*U_e = \bigcup_{q \in  U_e} T_q^*Q,
\]
and by $\tau^*_{U_e}:T^*U_e\to U_e$ the 
corresponding projection. Now, we consider the Riemannian metric on $U_e$
\[
g_e = 2 (e - V_{| U_e}) g_{| U_e}
\]
and the kinetic Hamiltonian function on $T^*U_e$ given by
\[
\kappa_{g_e}(\alpha _q) = \frac{1}{2}\|\alpha _q\|_{g_e}^2 = \frac{1}{4(e-V(q))} \|\alpha _q\|_g^2=\frac{1}{2(e-V(q))} \kappa_g(\alpha_q),\qquad \mbox{for }\alpha_q\in T^*_qU_e.
\]
$g_e$ is the Jacobi metric and the spheric cotangent bundle of radius 1 associated with $g_e$ is 
\[
\mathbb{S}_{g_e}(T^*U_e)=\kappa ^{-1}_{g_e}(\frac{1}{2})=\{ \alpha _q\in T^*U_e \; : \; \kappa _g(\alpha _q) =e-V(q) \} = H^{-1}_{(g,V)}(e)
\]
is a submanifold of codimension 1 of $T^*U_e$ and a bundle over $U_e$ with fiber by the point $q\in U_e$, the sphere of center $0_q\in T^*_qQ$ and radius $\sqrt{2(e-V(q))}$, with respect to the metric $g$. Moreover, 
\[
X_{\kappa_{g_e}}{}_{|\mathbb{S}_{g_e}(T^*U_e)}=\frac{1}{2( e-V\circ \tau^*_{U_e})}X_{H_{(g,V)}|\mathbb{S}_{g_e}(T^*U_e)}.
\]
Thus, the trajectories of the Hamiltonian system $(U_e, H_{(g,V)}{}_{|T^*U_e})$ associated with solutions on 
$\mathbb{S}_{g_e}(T^*U_e)=H_{(g,V)}^{-1}(e)$ are, up to reparametrization, geodesics of the Riemannian metric $g_e$ on $U_e$. In fact, if $c\colon I\to U_e$ and $c_e\colon I_e\to U_e$, with $0\in I\cap I_e$, are trajectories of the corresponding systems with the same initial condition, then there is a strictly increasing reparametrization $h\colon I\to I_e$ such that $c(s)=c_e(h(s)),$
\[
\frac{dh}{ds}=2(e-V\circ c )\qquad\mbox{ and } \qquad h(0)=0.
\]

On the other hand, $\mathbb{S}_{g_e}(T^*U_e)$ is a contact manifold and the corresponding Reeb vector field is just 
\[
{\mathcal R}=X_{\kappa_{g_e}}{}_{|\mathbb{S}_{g_e}(T^*U_e)}=\frac{1}{2(e-V\circ \tau^*_{U_e})}X_{H_{(g,V)}|\mathbb{S}_{g_e}(T^*U_e)}.
\]
\begin{remark}
An extension of the Maupertuis principle for a mechanical Hamiltonian function $H_{(g,V)}:= \kappa _g+V\circ \tau^* _Q$, where $g$ is Lorentz metric, has been considered, for instance, in  \cite{Lorentz}.
\end{remark} 

\subsection{The case of the harmonic oscillator}
We consider the Hamiltonian function on the cotangent bundle $T^*\R^2$ 
\[
H:T^*\R^2\to \R,\;\;\; H(q,p)=\frac{1}{2}(\sum_{i=1,2}p_i^2 +\sum_{i=1,2}q_i^2).
\]
We take an energy level $e> 0$ and we denote by $U_e$ the open subset 
$$U_e=\{(q_1,q_2)\in \R^2\, : \,q_1^2+ q_2^2<{e}\}.$$

In this case, the level set of $H$ for $e$ is
$$H^{-1}(e)=\{(q_1,q_2,,p_1,p_2)\in U_e\times \R^2\, : \, p_1^2+p_2^2+(q_1^2+ q_2^2)=2e\},$$
 which can be described as the level set for the value $\frac{1}{2}$ of the kinetic energy for the Riemannian metric on $U_e$
\[
g_e= 2(e-q_1^2-q_2^2) (dq_1\otimes dq_1+dq_2\otimes dq_2).
\]

The trajectories on $H^{-1}(e)$ of the system $(U_e, H_{|U_e\times \R^2})$ are 
\[
t\mapsto c(t)=(A_1\sin t + B_1\cos t, A_2\sin t + B_2\cos t)
\]
with $A_1^2+A_2^2+B_1^2+B_2^2=e$. Note that all of them are periodic. Moreover, we deduce that they are, up to reparametrization, geodesics of the Riemannian metric $g_e$ on $U_e$.

\subsection{The case of the hyperbolic plane}
We consider the hyperbolic plane  
$${\mathbb H}=\{(x,y)\in \R^2\, : \, y>0\}=\R\times \R^+$$ 
with the metric 
\[
g=\frac{1}{y^2}(dx\otimes dx + dy\otimes dy).
\]

Let $\Gamma$ be the discrete group of isometries generated by the isometry $\varphi_a:{\mathbb H}\to {\mathbb H}$ given by  $\varphi_a(x,y)=(a^2x,a^2y)$ with $a\in {\mathbb N}$,
$a\neq 0,1$. This group acts on ${\mathbb H}$ and  ${\mathbb H}/\Gamma$ is a $2$-dimensional compact Riemannian manifold of constant curvature $-1$, with the metric $\tilde{g}$ induced by $g$ on  ${\mathbb H}/\Gamma.$

If $V:{\mathbb H}\to \R$ is an invariant function with 
respect to $\Gamma$, on the cotangent bundle 
$T^*({\mathbb H}/\Gamma)$ we can consider the 
Hamiltonian function
\[
H:T^*({\mathbb H}/\Gamma)\to \R,\;\;\; H(z,\alpha_z)=\frac{1}{2}\|\alpha_z\|^2_{\tilde{g}} +\tilde{V}(z),
\]
with $\tilde{V}: {\mathbb H}/\Gamma\to \R$ the function characterized by 
$$\tilde{V}\circ \pi={V},$$
$\pi:{\mathbb H}\to {\mathbb H}/\Gamma$ being the 
canonical projection.

In this case, if $e$ is a positive real number such that $\tilde{V}$ is bounded by $e$, then  the $3$-dimensional submanifold $H^{-1}(e)$ is just a $S^1$-bundle on ${\mathbb H}/\Gamma$ and it can be described as the level set for the value $\frac{1}{2}$ of the kinetic energy for the Riemaniann metric on ${\mathbb H}/\Gamma$
\[
\tilde{g}_e=2(e-\tilde{V})\tilde{g},
\]
induced by the Jacobi metric $g_e=2(e-V)g$ on ${\mathbb H}$. 

The sectional curvature $c_{g_e}(x,y)$ at the point $(x,y)\in \mathbb{H}$ of the Riemannian manifold $({\mathbb H},g_e)$ is just 
\[
c_{g_e}(x,y)= \frac{1}{2}\Big (- y^2+y^4(\frac{\partial^2V}{\partial x^2} + \frac{\partial^2V}{\partial y^2})\Big ) =\frac{1}{2} ( -y^2+y^4\Delta(V)).
\]

Therefore, if $\Delta(V)<\frac{1}{y^2}$, we have that $({\mathbb H}/\Gamma,\tilde{g_e})$ is a compact Riemannian manifold with negative curvature.  

The trajectories of the system $({\mathbb H}/\Gamma, H)$ on the level set  $H^{-1}(e)=\mathbb{S}_{\tilde{g}_e}(T^*({\mathbb H}/\Gamma))$ are, up to reparametrization, geodesics of the Riemannian manifold $({\mathbb H}/ \Gamma ,{g_e})$. Using the results of Anosov \cite{Anosov}, we deduce that  the periodic trajectories determine an everywhere dense set in  $H^{-1}(e)=S_{\tilde{g}_e}(T^*({\mathbb H}/\Gamma))$ and the system is ergodic.

\section{The more general case of a linear Poisson structure}
In this section, we will extend the results in the previous section for mechanical Hamiltonian systems with respect to linear Poisson structures on a vector bundle. 

Let $\tau _A: A \to Q$ be a vector bundle of rank $n$
endowed with a linear Poisson structure on the dual 
vector bundle $\tau_{A^*}:A^*\to Q$, i.e. the 
corresponding Lie bracket $\{\cdot,\cdot\}_{A^*}$ of 
functions on $A^*$ satisfies that the bracket of
two fiberwise linear functions is again a fiberwise linear 
function. Note that there exists a one-to-one 
correspondence between the set of fiberwise linear 
functions on $A^*$ and the space $\Gamma(A)$ of sections of $\tau _A: A \to Q$ which is defined 
by
\[
\begin{array}{rcl}
\hat{ } \; \colon \Gamma (A) &\mapsto  & 
\{ \mbox{ fiberwise linear functions on } A^* \}\\ X & \mapsto & \hat{X}
\end{array}
\]
with
\[
\hat{X}(\alpha _q)=\langle \alpha _q, X(q)\rangle ,\qquad \mbox{for } \alpha_q\in A^*_q \mbox{ and } q\in Q.
\]
Thus, for the linear Poisson structure $\{\cdot,\cdot\}_{A^*}$ on $A^*$, we have that 
\[
\{ \hat{X},\hat{Y} \} _{A^*} =-\widehat{\lcf X, Y\rcf}, \qquad \mbox{for }X, Y\in \Gamma (A),
\]
with $\lcf X, Y\rcf \in \Gamma(A)$.
\begin{remark}\label{rem:Liealg}
The Poisson structure $\{\cdot,\cdot\}_{T^*Q}$ associated with the canonical symplectic structure $\omega _Q$ on $T^*Q$ is linear. In fact, if $X,Y\in \mathfrak{X}(Q)$ then
\[
\{ \hat{X},\hat{Y} \} _{T^*Q} =-\widehat{[X, Y]},
\]
where $[X,Y]\in \mathfrak{X}(Q)$ is the standard Lie bracket of vector fields on $Q$.

More generally, linear Poisson structures on a vector bundle $A^*$ are in one-to-one correspondence with Lie algebroid structures on the dual vector bundle $A$ (see \cite{Courant}). We recall that a Lie algebroid structure $(\lcf \cdot ,\cdot \rcf , \rho )$ on a vector bundle $A\to M$ is a Lie algebra bracket $\lcf \cdot ,\cdot \rcf$ on $\Gamma (A)$ and a bundle map $\rho \colon A\to TM$, called the anchor map, such that
\[
\lcf X, fY\rcf =f\lcf X, Y\rcf + \rho (X)(f)Y,  \; \; \mbox{ for } X, Y\in \Gamma(A), f\in C^\infty(M).
\]
\end{remark}
For a linear Poisson structure $\{\cdot,\cdot\}_{A^*}$ on $A^*$, it follows that   
\[
\{f\circ \tau_{A^*},\hat{X} \}_{A^*}=\rho (X)(f)\circ \tau _{A^*},
\qquad 
\{f\circ \tau_{A^*}, h\circ \tau_{A^*}\}_{A^*}=0,
\]
for $f,h\in C^\infty(Q)$ and $X\in \Gamma (A)$ (see, for instance, \cite{Courant}).

So, if $(q^i,y_\alpha )$ are local coordinates on $A^*$, associated with local coordinates $(q^i)$ on $Q$ and local coordinates $(y_\alpha )$ coming from a local basis $\{ e^\alpha \}$ of sections of $A^*$, then
\[
\hat{e}_\alpha=y_\alpha
\]
where $\{e_\alpha\}$ is the dual basis of sections of $A$. Moreover, if 
\[
\lcf e_\alpha ,e_\beta \rcf= C_{\alpha \beta}^\gamma e_\gamma, \qquad \rho (e_\alpha )=\rho ^i_\alpha \frac{\partial }{\partial q^i},
\]
we have that 
\[
\{ y_\alpha ,y_\beta \}_{A^*} =-C_{\alpha \beta}^\gamma y_\gamma, \qquad \{ q^i,y_\alpha \}_{A^*}=\rho^i_\alpha,\qquad \{ q^i ,q^j\}_{A^*} =0,
\]
and the Poisson 2-vector $\Pi _{A^*}$, associated with the linear Poisson bracket $\{\cdot ,\cdot \}_{A^*}$, has local expression
\begin{equation}\label{Local-expression-Poisson-2-vector}
\Pi _{A^*}(q,y)=\rho ^i_\alpha (q)\frac{\partial }{\partial q^i}\wedge \frac{\partial }{\partial y_\alpha}-\frac{1}{2} C_{\alpha\beta}^\gamma (q) y_\gamma \frac{\partial }{\partial y_\alpha}\wedge \frac{\partial }{\partial y_\beta}.
\end{equation}

The previous linearity conditions can be described in  term of the homogenity with respect to the  Liouville vector field $\Delta_{A^*}$ of $\tau_{A^*}$ which is locally defined by
\begin{equation}\label{Liou}
\Delta_{A^*}=\sum_\alpha y_\alpha \frac{\partial}{\partial y_\alpha}.
\end{equation}
Indeed, $F:A^*\to \R\in C^\infty (A^*)$ is fiberwise linear if and only if 
$\Delta_{A^*}(F)=F$.  Then, one can prove that the $2$-vector $\Pi_{A^*}$ on $A^*$ associated with a Poisson bracket on $A^*$ is linear   if and only if 
\[
{\mathcal L} _{\Delta_{A^*}}\Pi_{A^*} =
-\Pi_{A^*}.
\]
Now, suppose that we have  a bundle metric $g$ on $A$. Consider the Hamiltonian function 
$\kappa_{(A,g)}: A^*\to {\R}$ given by 
\[
\kappa_{(A,g)}(\alpha_q)=\frac{1}{2} \|\alpha_q\|_g^2, \mbox{ for all } \alpha_q\in A^*_q \mbox{ and } q\in Q,\]
where $\|\cdot\|_g$ is the norm on $A^*$ induced 
by $g$. If $(g^{\alpha \beta})$ is the matrix of the bundle metric $g$ on $A^*$ with respect to a local basis $\{ e^\gamma \}$ of $\Gamma (A^*)$, then the local expression of $\kappa _{(A,g)}$ is
\begin{equation}\label{3.4'}
\kappa_{(A,g)} (q,y)=\frac{1}{2} g^{\alpha \beta }(q)y_\alpha y_\beta.
\end{equation}
Like in the classical case on $T^*Q$,  we consider the spheric dual bundle of radius 1
\[
{\mathbb S}_g(A^*)=\kappa_{(A,g)}^{-1}(\frac{1}{2}),
\]
which is a fiber bundle on $Q$. 

\medskip
 Extending the construction in Section 
 \ref{subsection2.1}, we will show that on  ${\mathbb S}_g(A^*)$ one can define a Jacobi structure.  
 
 \medskip
 We recall that a {\it Jacobi structure} on a manifold $M$ is 
 a pair $(\Lambda,E)$, with $\Lambda\in{\mathcal V}^2(M)$ a $2$-vector on $M$ and $E$ a vector field on $M$ such that
\[
[\Lambda,\Lambda]=2\Lambda\wedge E\mbox{ and } {\mathcal L}_E \Lambda =0.
\]
Here $[\cdot ,\cdot ]$ is the Schouten-Nijenhuis bracket. 
Every Jacobi structure $(\Lambda,E)$ induces a 
bracket  $\{\cdot,\cdot\}:C^\infty(M)\times C^\infty(M)\to 
C^\infty(M)$  on the space of real $C^\infty$-functions on $M$ given by 
\begin{equation}\label{Jacobi:bracket}\{f_1,f_2\}=\Lambda(df_1,df_2) + f_1 E(f_2) -f_2 E(f_1),\end{equation}
which is ${\R}$-bilinear, skew-symmetric, satisfies the Jacobi identity and it 
associates a first order differential operator with any function. Conversely, every bracket of functions on the manifold $M$ with these properties induces a Jacobi structure on $M$ \cite{kirillov, Li}. This bracket is called the {\it Jacobi bracket  on $M$}. 

\medskip
Contact structures are particular examples of Jacobi structures. In fact, if $\eta$ is a contact structure on $M$ then the vector bundle morphism $\flat _\eta \colon TM\to T^*M$ given by
\[
\flat _\eta (v)=i_vd\eta (\tau _M(v))+\langle \eta ( \tau _M(v)),v\rangle \eta (\tau_M(v)),
\]
for $v\in TM$, with $\tau_M\colon TM\to M$ the canonical projection, is a vector bundle isomorphism. So,
we can consider the 2-vector $\Lambda_\eta$ on $M$ defined by
\[
\Lambda _\eta (\alpha ,\beta )=-d\eta (\flat _\eta ^{-1} (\alpha ), \flat _\eta ^{-1} (\beta )), \qquad \mbox{for }\alpha ,\beta \in \Omega^1(M).
\]
Moreover, if $\mathcal{R}$ is the Reeb vector field, one may see that the couple  $(\Lambda _\eta ,E_\eta =-\mathcal{R})$ is a Jacobi structure on $M$ (see, for instance, \cite{Li}). 
\begin{remark}\label{contact-spheric}
Let $g$ be a Riemannian metric on $Q$, $\mathbb{S}_g(T^*Q)$ the spheric cotangent bundle of radius 1 and $\eta$ the canonical contact structure on $\mathbb{S}_g(T^*Q)$ (see Section \ref{subsection2.1}). Then, the Jacobi structure $(\Lambda_{\mathbb{S}_g(T^*Q)}, E_{\mathbb{S}_g(T^*Q)})$ on $\mathbb{S}_g(T^*Q)$ associated with $\eta$ is given by
\[
\Lambda_{\mathbb{S}_g(T^*Q)}= (\Pi _{T^*Q} +\Delta_{T^*Q} \wedge X_{\kappa_g} )_{|\mathbb{S}_g(T^*Q)}, \qquad 
E_{\mathbb{S}_g(T^*Q)}=-(X_{\kappa _g})_{|\mathbb{S}_g(T^*Q)},
\]
where $\Pi_{T^*Q}$ is the canonical Poisson 2-vector on $T^*Q$ and $X_{\kappa_g}=-i(d\kappa _g) \Pi_{T^*Q}$ is the Hamiltonian vector field of $\kappa _g$, with $\kappa _g\colon T^*Q\to \R$ the kinetic energy on $T^*Q$ induced by the Riemannian metric on $g$. In fact, if $X,Y\in \mathfrak{X}(T^*Q)$ then
\[
\omega _Q(X,Y)=\Pi _{T^*Q} (\flat _{\omega _Q}(X),\flat _{\omega _Q}(Y)),
\]
with $\flat _{\omega _Q}\colon \mathfrak{X}(T^*Q)\to \Omega ^1(T^*Q)$ the isomorphism of $C^\infty (T^*Q)$-modules given by $\flat_{\omega_Q}(X)=i_X\omega _Q$.
So, using that $i^\ast \omega _Q=-d\eta$, it follows that 
\[
-d\eta (X,Y)=\Pi _{T^*Q}(i_Xd\eta ,i_Yd\eta )=\Pi _{T^*Q} (\flat _\eta X,\flat _\eta Y)+\eta (Y) \Pi _{T^*Q} (\eta ,\flat _\eta X)-\eta (X) \Pi _{T^*Q} (\eta ,\flat _\eta Y).
\]
Now, since $i_{\lambda _Q}\Pi _{T^*Q}=\Delta_{T^*Q}$, we deduce that
\[
-d\eta (X,Y)=\Pi _{T^*Q} (\flat _\eta X,\flat _\eta Y)+(\flat _\eta X )( \Delta _{T^*Q} ) \eta (Y)  -(\flat _\eta Y )( \Delta _{T^*Q} ) \eta (X) .
\]
On the other hand, the Reeb vector field of $\mathbb{S}_g(T^*Q)$ is $\mathcal{R}= (X_{\kappa_g})_{|\mathbb{S}_g(T^*Q)}$ (see Section \ref{subsection2.1}). This implies that
\[
(\flat _\eta X)(X_{\kappa _g})=\eta (X)
\]
and, thus,
\[
-d\eta (X,Y)=( \Pi _{T^*Q}+ \Delta _{T^*Q} \wedge X_{\kappa _g})_{|\mathbb{S}_g(T^*Q)} (\flat _\eta X,\flat _\eta Y)
\]
which proves the result.
\end{remark}
The previous remark provides a good motivation to claim the following theorem.
\begin{theorem}\label{prop5.3}
Let $\tau_A:A\to Q$ be a vector bundle such that  $\Pi_{A^*}$ is a linear Poisson structure on $A^*$ and let  $g:A\times A\to {\mathbb R}$ be a bundle metric on  $A$.  Then:
\begin{enumerate}
\item [$(i)$] The pair $(\Lambda_{A^*},E_{A^*}) \in \mathcal{V}^2(A^*)\times
\frak X(A^*)$ given by
\begin{equation}\label{Jacobistructure}
\Lambda_{A^*} = \Pi_{A^*} + \Delta_{A^*}\wedge
X_{\kappa_{(A,g)}}, \qquad E_{A^*}=-
X_{\kappa_{(A,g)}}
\end{equation}
is a Jacobi structure on $A^*,$ where $\Delta_{A^*}$ is
 the Liouville vector field on the vector bundle $\tau_{A^*}:A^*\rightarrow M$ and
$X_{\kappa_{(A,g)}} \in {\mathfrak X}(A^*)$ is the Hamiltonian vector field of ${\kappa_{(A,g)}}\in C^\infty(A^*)$ with respect to the linear Poisson structure $\Pi_{A^*}$, i.e.
$$
X_{\kappa_{(A,g)}}=-i(d{\kappa_{(A,g)}})\Pi_{A^*}.
$$

\item[$(ii)$] The Jacobi structure $(\Lambda_{A^*}, E_{A^*})$ induces  a Jacobi structure $(\Lambda_{\mathbb{S}_g(A^*)},E_{{\mathbb{S}_g(A^*)}})$ on ${\mathbb{S}_g(A^*)}$ such that,  if $\mathcal{O}_Q$ is  the zero section of $A^*$, the map
\[
\Psi:{{\mathbb{S}}_g(A^*)}\times\R \to A^*-\mathcal{O}_Q,\;\;\; \Psi(\beta_q,t)=e^t\beta_q
\]
is a Poisson isomorphism when we consider on $A^*-\mathcal{O}_Q$ the induced Poisson structure by $\Pi_{A^*}$ and on ${\mathbb{S}_g(A^*)}\times \R$ the Poissonization of the Jacobi structure on 
${\mathbb{S}_g(A^*)},$ i.e.
\[
e^{-t}\Big ( \Lambda_{\mathbb{S}_g(A^*)} - E_{\mathbb{S}_g(A^*)}\wedge \frac{\partial }{\partial t}\Big ).
\]
\item[$(iii)$] The Jacobi bracket $\{\cdot,\cdot\}_{\mathbb{S}_g(A^*)}$ on $\mathbb{S}_g(A^*)$  and the Poisson bracket $\{\cdot,\cdot\}_{A^*}$
on $A^*$ are related as follows
\[
\begin{array}{rcl}
\{G_1\circ \iota,G_2\circ \iota\}_{\mathbb{S}_g(A^*)} & =& \big(\{G_1,G_2\}_{A^*} + 
\big\{\kappa_{(A,g)},G_1\}_{A^*}(G_2-\Delta_{A^*}(G_2))-\\&&\big\{\kappa_{(A,g)},G_2\}_{A^*}(G_1-\Delta_{A^*}(G_1))\big)\circ \iota
\end{array}
\]
for all $G_1,G_2 \in C^\infty (A^*)$, with $\iota:\mathbb{S}_g(A^*)  \to A^*$ the inclusion map. 
\end{enumerate}
\end{theorem}

\begin{proof} Our theorem is a consequence of some results which where proved in \cite{GM} (see also \cite{DLM,GrIgMaPaUr}). Anyway, in order to make our paper self-contained, we will give a simple proof of the theorem.

{\it (i)} It is well-known  that $X_{\kappa_{(A,g)}}$ is an infinitesimal automorphism of the Poisson structure $\Pi_{A^*}$, that is,  
\[
{\mathcal L}_{X_{\kappa_{(A,g)}}}\Pi_{A^*}=0.
\]
Then, \cite[Cor.1]{GM} implies that
$\Lambda_{A^*} = \Pi_{A^*} + \Delta_{A^*}\wedge
X_{\kappa_{(A,g)}}$, $E_{A^*}=-
X_{\kappa_{(A,g)}}$ is a Jacobi structure on $A^*$.  

\medskip 

$(ii)$ Using \eqref{Liou} and \eqref{3.4'}, we have that
$\Delta_{A^*}({\kappa_{(A,g)}})=2{\kappa_{(A,g)}}$,
that is, ${\kappa_{(A,g)}}$ is homogeneous of degree 2. Therefore, from \cite[Cor.2]{GM},  $\Lambda_{A^*}$ and $E_{A^*}$ restrict to ${\mathbb{S}_g(A^*)}$. 

On the other hand, note that the diffeomorphism $\Psi$ is just the restriction of the flow of $\Delta _{A^*}$ to ${\mathbb{S}_g(A^*)}$. Thus, we have that 
\[
T\Psi \circ \frac{\partial }{\partial t}=\Delta _{A^*} \circ \Psi.
\]
This, using \eqref{Jacobistructure}, implies that $\Psi$ is a Poisson isomorphism between ${\mathbb{S}_g(A^*)}\times \mathbb{R}$ and $A^*-\mathcal{O}_Q$.

\medskip

$(iii)$ It is a consequence from (\ref{Jacobi:bracket}) and (\ref{Jacobistructure}).
\end{proof}
\begin{remark}
From \eqref{Jacobistructure}, we deduce that the Jacobi structure $(\Lambda _{A^*},E_{A^*})$ on $A^*$ is polynomical of degree 3.
\end{remark}

Now, let $g$ be a bundle metric on $A$ and $V\in C^\infty (Q)$.  Consider the Hamiltonian function 
$H_{(A,g,V)}: A^*\to \R$ of mechanical type given by 
\begin{equation}\label{H-g-V}H_{(A,g,V)}(\alpha_q)=\kappa_{(A,g)}(\alpha_q)+V(q), \mbox{ for all } \alpha_q\in A^*_q \mbox{ and } q\in Q.\end{equation}

Suppose that $e \in \mathbb{R}$ is such that $U_e = \{ q \in Q \, : \,  V(q)<e \}$
is a non-empty subset of $Q$.  Denote by
\[
A_e = \bigcup_{q \in  U_e} A_q
\]
the corresponding vector bundle over $U_e$ and by $A^*_e$  its  dual bundle. 
Now, we consider the Jacobi bundle metric $g_e$ on $A_e$ defined by
\[
g_e = 2(e - V_{|U_e})g_{|A_e}
\]
and the corresponding Hamiltonian function $\kappa_{(A_e,g_e)}: A^*_e \to \mathbb{R}$ given by
\begin{equation}\label{H-g-e}
\kappa_{(A_e,g_e)}(\alpha_q) = \frac{1}{2}\|\alpha_q\|_{g_e}^2 = \frac{1}{4(e-V(q))} \|\alpha_q\|_g^2=\frac{1}{2(e-V(q))}\kappa_{(A_e,g)}(\alpha _q) \; \; \end{equation}
for $\alpha_q \in (A^*_e)_q$,  with  $q \in U_e.$
On the other hand, we may introduce the spheric dual bundle  on $U_e$ defined by
\begin{equation}\label{S-estrella-e}
{\mathbb S}_{g_e}(A^*) = \kappa_{(A_e,g_e)}^{-1}(\frac{1}{2}) = \bigcup_{q \in U_e}\{\alpha_q \in A^*_q \, : \, \|\alpha_q\|^2_{g} = 2(e- V(q))\}= H_{(A_e,g, V)}^{-1}(e).
\end{equation}
Then, one has the following result

\begin{theorem}\label{estructura-S-estrella-e}  
\begin{enumerate}
\item[$(i)$]
 ${\mathbb S}_{g_e}(A^*)$ is a submanifold of codimension $1$ in $A^*_e$ and if $\alpha_q \in {\mathbb S}_{g_e}(A^*)$
\begin{equation}\label{tangente}
\begin{array}{rcl}
T_{\alpha_q} {\mathbb S}_{g_e}(A^*)  &=& \{ w \in T_{\alpha _q} A^*_e \, : \, \langle dH_{(A_e,g, V)}(\alpha_q), w \rangle  = 0 \} \\&=& \{w \in T_{\alpha_q} A^*_e \, : \, \langle d\kappa_{(A_e,g_e)}(\alpha_q), w \rangle  = 0 \}.
\end{array}
\end{equation}

\item[$(ii)$] ${\mathbb S}_{g_e}(A^*) $ is a bundle over $U_e$ with fiber by $q \in U_e$ the sphere of center $0_q \in A^*_q$ and radius $\sqrt{2(e - V(q))}$, with respect to the bundle metric $g$.

\item[$(iii)$] The Hamiltonian vector fields of $H_{(A,V,g)}$ and $\kappa_{(A,g_e)}$ are tangent 
to ${\mathbb S}_{g_e}(A^*)$ and
\[
X_{H_{(A,g,V)}|{\mathbb S}_{g_e}(A^*)}=2(e-V_{|U_e})X_{\kappa_{(A,g_e)}|{\mathbb S}_{g_e}(A^*)}.
\]

\item[$(iv)$] The bundle ${\mathbb S}_{g_e}(A^*)$ admits a Jacobi bundle structure $(\Lambda_{{\mathbb S}_{g_e}(A^*)},E_{{\mathbb S}_{g_e}(A^*)})$ with
$$
\begin{array}{rcl}
\Lambda_{{\mathbb S}_{g_e}(A^*)}&=&\Pi_{A^*| {\mathbb S}_{g_e}(A^*)}+\displaystyle\frac{1}{2(e-V_{|U_e})}\Delta_{A^*|{\mathbb S}_{g_e}(A^*)}\wedge X_{H_{(A,g,V)}|{\mathbb S}_{g_e}(A^*)},\\[5pt]
E_{{\mathbb S}_{g_e}(A^*)}&=&-\displaystyle\frac{1}{2(e-V_{|U_e})}X_{H_{(A,g,V)}|{\mathbb S}_{g_e}(A^*)}.
\end{array}$$

 \item[$(v)$] If  the curve $c:I\to A_e^*$ (resp., $c_e:I_e\to A_e^*$), with 0 belonging to the interval $I$ (resp. $I_e$),  is the trajectory of the system $(A_e^*, H_{(A,g,V)})$  (resp. $(A_e^*, \kappa_{(A_e,g)}))$  such that $c(0)=\alpha_q$ (resp. $c_e(0)=\alpha_q$), then there is a strictly increasing reparametrization $h: I\to I_e$ such that 

$$c(s)=c_e(h(s)).$$

Moreover, 
$$\frac{dh}{ds}=2(e-V\circ\tau_{A^*}\circ c),\;\;\;\; h(0)=0.$$

\end{enumerate}
\end{theorem}
\begin{proof} Since 
\[
\langle dH_{(A_e,g, V)}(\alpha_q), \Delta_{A^*}(\alpha_q) \rangle  =\|\alpha_q\|_g^2= 2(e - V(q)) >  0,
\]
and
\[
\langle d\kappa_{(A_e,g_e)}(\alpha_q), \Delta_{A^*}(\alpha_q) \rangle  = \|\alpha_q\|^2_{g_e} = 1 >  0,
\]
then
\[
dH_{(A,g, V)}(\alpha_q)\neq 0\mbox{ and }d\kappa_{(A_e,g_e)}(\alpha_q) \neq 0.
\]
Thus, ${\mathbb S}_{g_e}(A^*)$ is a submanifold of $A^*_e$ of codimension $1$ and (\ref{tangente}) holds. Moreover, 
\[
T_{\alpha_q} A^*_e = T_{\alpha_q} {\mathbb S}_{g_e}(A^*)  \oplus \langle \Delta_{A^*}(\alpha_q)\rangle .
\]
Therefore, using that $\Delta_{A^*}$ is vertical with respect to the projection $\tau^*_{|A^*_e}: A^*_e \to U_e$, it follows that the restriction of $\tau^*_{|A^*_e}$ to ${\mathbb S}_{g_e}(A^*) $ is a fibration
\[
\tau^*_{|{\mathbb S}_{g_e}(A^*) }: {\mathbb S}_{g_e}(A^*)  \to U_e.
\]
In addition, it is easy to prove that the fiber of $\tau^*_{|{\mathbb S}_{g_e}(A^*) }$ by $q \in U_e$ is just the sphere of center $0_q \in A^*_q$ and radius $\sqrt{2(e-V(q))}$, with respect to the bundle metric $g$. This proves {\it i)} and {\it ii)}. 

Now, using \eqref{tangente}, it is clear that the restriction to ${\mathbb S}_{g_e}(A_e^*)$ of the Hamiltonian vector fields $X_{H_{(A_e,g,V)}}$ and 
$X_{\kappa_{(A_e,g_e)}}$ is tangent to ${\mathbb S}_{g_e}(A_e^*)$. In addition, if $\alpha _q\in {\mathbb S}_{g_e}(A_e^*)$, we have that
\[
\kappa _g (\alpha _q)=e-V(q),
\]
and, from (\ref{H-g-e}), we deduce that
\[
d\kappa _{(A_e,g_e)}(\alpha _q)=\frac{1}{2(e-V(q))}dH_{(A_e,g,V)}(\alpha _q).
\]
This implies that 
\[
(X_{H_{(A_e,g,V)}})_{| \mathbb{S}_{g_e}(A^*_e)}=2(e-V_{| U_e})(X_{\kappa_{(A_e,g_e)}})_{| \mathbb{S}_{g_e}(A^*_e)},
\]
which proves {\it (iii)}.

{\it (iv)} follows using {\it (iii)} and Theorem \ref{prop5.3}.

Finally, {\it (v)} is a consequence of {\it (iii)}.
\end{proof}
\begin{remark}\label{results-Section-2}
If we apply Theorem \ref{estructura-S-estrella-e} to the particular case when $A=TQ$ and the linear Poisson structure on $T^*Q$ is the canonical symplectic structure then, using Remark \ref{contact-spheric}, we deduce all the results in Section \ref{subsection2.1}.
\end{remark}
If $(q^i, y_\alpha )$ are local coordinates on $A^*$ then, using \eqref{Local-expression-Poisson-2-vector}, \eqref{3.4'} and \eqref{H-g-V}, we deduce that
\begin{equation}\label{3.15'}
\begin{array}{rcl}
X_{\kappa_{(A_e,g_e)}}=\displaystyle \frac{1}{2(e-V)} X_{(H,g,V)}&=&\displaystyle 
\frac{1}{2(e-V)}\Big (  ( \rho ^i_\alpha g^{\alpha\beta} y_\beta ) \frac{\partial}{\partial q^i}- (\frac{1}{2} 
\frac{\partial g^{\alpha\gamma}}{\partial q^i} \rho ^i_\beta y_\alpha y_\gamma \\[8pt]  &&\displaystyle -C^\gamma_{\alpha \beta}g^{\alpha \gamma}y_\gamma y_\mu +\frac{\partial V}{\partial q^i}\rho ^i_\beta  )\frac{\partial }{\partial y_\beta}\Big ).
\end{array}
\end{equation}
\begin{remark}
Note that if in Theorems \ref{prop5.3} and \ref{estructura-S-estrella-e} we replace the bundle metric $g$ by a non-degenerate symmetric section of the vector bundle $A^*\otimes A^*\to Q$ then both results remain valid. This implies that the results in this section can also be applied to the dynamics of relativistic particles. In such a case, we have a Lorentz bundle metric.
\end{remark}

\section{Examples}
\label{section6} \setcounter{equation}{0}
\medskip

In this section we will apply the previous constructions to some Hamiltonian mechanical systems with respect to linear Poisson structures on vector bundles.

\subsection{Mechanical Hamiltonian systems on the dual bundle of an Atiyah Lie algebroid 
associated with a trivial principal bundle}
Let $G$ be a Lie group with Lie algebra $\g$ and consider the product manifold $M=G\times Q$, with $Q$ a smooth manifold. Then, $M$ is the total space of a trivial principal $G$-bundle 
\[
p\colon M=G\times Q\rightarrow Q ,
\]
where the action of $G$ on $M$ is just by left-translation on the first factor. 

The cotangent lift of the $G$-action on $M$ induces a 
principal action of $G$ on $T^*M=T^*G\times T^*Q$. 
Indeed, using the left-trivialization of $T^*G$, we can 
identify  $T^*M$ with the product manifold $(G\times \g^*)
\times  T^*Q$ and, under this identification, we have that 
the quotient space $T^*M/G$  is just the trivial vector 
bundle (over $Q$), $T^*Q\times \g^*\to Q$ (see, for 
instance, \cite{Abraham-Marsden}). 

In the total space of the vector bundle  $T^*Q\times \g^*$ there is a fiberwise linear 
Poisson structure, denoted by $
\Pi _{T^*Q\times \g^*}$, which is the product of the canonical 
symplectic structure on $T^*Q$ and the Lie-Poisson 
structure on $\g^*$,  whose bracket is characterized as follows 
\[
\{f_1,f_2\}_{{\mathfrak g}^*}(\alpha)=-\alpha([df_1(\alpha),df_2(\alpha)]_{\mathfrak g}),\;\;\;\; \mbox{ for all } \alpha\in {\mathfrak g}^* \mbox{ and } f_1,f_2\in C^\infty({\mathfrak g}^*).
\]
Here, $[\cdot ,\cdot ]_{\mathfrak g}$ is the Lie bracket on $\mathfrak{g}$. If $(q^i, p_i)$ are local fibred coordinates on $T^*Q$ and $\{ e_a\}$ is a basis of $\g$, we have the corresponding local coordinates $(q^i,y_\alpha)=(q^i, p_i,y_a)$ on $T^*Q\times \g^*$ and, for every $\varphi ,\psi\in C^\infty (T^*Q\times \g^*)$
\[
\{\varphi, \psi \}_{T^*Q\times \g^*} = \Big ( \frac{\partial \varphi }{\partial q^i} \frac{\partial \psi }{\partial p_i} -\frac{\partial \varphi }{\partial p_i} \frac{\partial \psi }{\partial q^i} \Big ) -\frac{1}{2}  c_{ab}^c y_c \Big ( \frac{\partial \varphi }{\partial y_a} \frac{\partial \psi }{\partial y_b} -\frac{\partial \varphi }{\partial y_b} \frac{\partial \psi }{\partial y_a} \Big ) ,
\]
where $c^{c}_{ab}$ are the structure constants of the Lie algebra $\mathfrak{g}$ with respect to the basis $\{e_a\}$.
\begin{remark}
The canonical symplectic structure on $T^*M$ is $G$-invariant. Thus, it induces a Poisson structure on the quotient space $T^*M/G\cong T^*Q\times \g^*$. This Poisson structure is just $\Pi_{T^*Q\times g^*}$ (see, for instance, \cite{MMOPS}). 
\end{remark}
Now, let $g$ be a bundle metric on the vector bundle $T^*Q\times \g^*\to Q$ and $V\in C^\infty (Q)$ a real $C^\infty$-function on $Q$. Then, we can consider the mechanical Hamiltonian function $H_{(g,V)}\colon T^*Q\times \g^*\to \R$ given by
\begin{equation}\label{H-g-V-Atiyah}
H_{(g,V)}(\alpha _q,\gamma) = \frac{1}{2} \| (\alpha _q,\gamma )\| ^2_g +V(q), \mbox{ for }\alpha_q\in T_q^*Q \mbox{ and } \gamma\in \g^*.
\end{equation}
\begin{remark}
Hamiltonian functions as in \eqref{H-g-V-Atiyah} may be obtained by reduction, from $G$-invariant Riemannian metrics and $G$-invariant potential functions on $M=G\times Q$.
\end{remark}
If 
\[
\left (
\begin{array}{cc}
g^{ij}(q) & g^{ia} (q) \\
g^{ia}(q) & g^{ab} (q)
\end{array}
\right )
\]
is the matrix of coefficients of the scalar product $g(q)$ on $T^*_qQ\times \g^*$ then it follows that 
\[
H_{(g,V)} (q^i,p_i,y_a)= \frac{1}{2} \displaystyle ( g^{ij}(q) p_ip_j+g^{ab}(q)y_ay_b ) +g^{ia}(q)p_iy_a +V(q).
\]
Next, let $e$ be a real number such that $U_e=\{q\in Q\, :\, V(q) < e\}\not=\emptyset$. Then 
\[
H_{(g,V)}^{-1}(e)=\{\alpha_q\in T^*U_e\, :\, \|\alpha_q\|^2_g=2(e-V(q))\}=\kappa_{g_e}^{-1}(\frac{1}{2})=\mathbb{S}_{g_e}(T^*U_e\times \g^*),
\]
where $\kappa_{g_e}\colon T^*U_e\times \g^*\to \R$ is the kinetic energy associated with the Jacobi bundle metric on $T^*U_e\times \g^*$ given by
\[
g_e=2({e-V\circ\tau^*_Q})g.
\]
Using  \eqref{3.15'}, it follows that the local expression of the Hamiltonian vector field on $T^*U_e\times \g^*$ of the kinetic energy $\kappa _{g_e}$ is 
\begin{eqnarray*}
X_{\kappa_{g_e}} (q,p,y)&=&\frac{1}{2(e-V)} \left ( \big (g^{i\alpha}(q)y_\alpha \big )\displaystyle \frac{\partial }{\partial q^i}- \Big (\frac{1}{2}\frac{\partial g^{\alpha\gamma}}{\partial q^i}y_\alpha y_\gamma +\frac{\partial V}{\partial q^i}\Big ) \displaystyle \frac{\partial }{\partial p_i} \right .\\&& \left . +(c^{c}_{ab}g^{ia}(q)p_iy_c + c_{ab}^cg^{ad}(q)y_cy_d)\displaystyle \frac{\partial }{\partial y_b}\right ) .
\end{eqnarray*}
In addition, the Jacobi structure on $\S _{g_e}(T^*U_e\times \g^*)$ is the restriction to $\S _{g_e}(T^*U_e\times \g^*)$ of the Jacobi structure $(\Lambda _{T^*U_e\times \g^*},E_{T^*U_e\times \g^*})$ on $T^*U_e\times \g^*$ given by
\[
\begin{array}{rcl}
\Lambda _{T^*U_e\times \g^*}&=&\displaystyle  \frac{\partial  }{\partial q^i}\wedge \frac{\partial }{\partial p_i} - \frac{1}{2}  c_{ab}^c y_c  \frac{\partial }{\partial y_a} \wedge \frac{\partial }{\partial y_b} + y_\alpha \frac{\partial  }{\partial y_\alpha } \wedge X_{\kappa _{g_e}},
\\
E_{T^*U_e\times \g^*}&=&-X_{\kappa _{g_e}}.
\end{array}
\]
\begin{remark}
If $Q$ is a single point, then the Hamiltonian function is the kinetic energy
\[
\kappa_{\langle \cdot, \cdot \rangle}:{\mathfrak g}^*\to \R,\;\;\; \kappa _{\langle \cdot, \cdot \rangle}(\alpha)=\frac{1}{2}\langle \alpha,\alpha\rangle^*,
\]
associated with a scalar product ${\langle \cdot, \cdot \rangle}$ on $\mathfrak{g}$. Here,
$\langle \cdot,\cdot\rangle ^*: {\mathfrak g}^*\times {\mathfrak g}^*\to {\mathbb R}$ denotes the corresponding scalar product on ${\mathfrak g}^*$. In this setting, the Jacobi structure $(\Lambda_{{\mathfrak g}^*}{}_{|{\mathbb{S}_{\langle \cdot, \cdot \rangle} (\mathfrak{g}^*)}},E_{{\mathfrak g}^*}{}_{|\mathbb{S}_{\langle \cdot, \cdot \rangle} (\mathfrak{g}^*)})$ on 
\[
\mathbb{S}_{\langle \cdot, \cdot \rangle} (\mathfrak{g}^*)=\kappa_{\langle \cdot, \cdot \rangle}^{-1}\big ( \frac{1}{2}\big )=\{\alpha\in {\mathfrak g}^* \, : \,\langle \alpha,\alpha\rangle ^*=1\},
\]
is $p$-projectable on the conformal Jacobi structure introduced by Lichnerowicz in \cite{Li1} in the quotient space $(\mathfrak{g}^*-\{ 0\})/ (\Delta _{\mathfrak{g}^*})_{|\mathfrak{g}^*-\{0\}}$, 
$p\colon \mathbb{S}_{\langle \cdot, \cdot \rangle} (\mathfrak{g}^*)\to (\mathfrak{g}^*-\{ 0\})/ (\Delta _{\mathfrak{g}^*})_{|\mathfrak{g}^*-\{0\}}$ being the canonical covering map.

An interesting example of this situation is given by the rigid body, in which $\mathfrak{g}$
is the Lie algebra ${\mathfrak{so}}(3)$ of the orthogonal group $SO(3)$.
\end{remark}
\begin{example}[{\bf The coupled planar pendula}]
{\rm As a simple example of a mechanical Hamiltonian system on the dual bundle of an Atiyah Lie algebroid, we will consider the coupled planar pendular \cite{MMT}. Let us consider two pendulums which are coupled moving under the influence of a potential depending on the hinge angle between them. Denote by $\theta _1$ and $\theta _2$ the angles formed by the straight lines through the joint and their centers of mass relative to an inertial coordinate system fixed in space. Thus, the configuration space is $M=\mathbb{T}^2=S^1\times S^1$.  The metric in this case is just
\[
g= \frac{1}{2} (d\theta _1 \otimes d\theta _1+d\theta _2 \otimes d\theta _2).
\]
and the potential energy is of the form $V(\theta _1-\theta _2)$. There is an action of $S^1$ on $M$
\[
\phi:S^1\times M\to M
\]
which, considering coordinates $(\varphi , \psi)$ given by
\[
\varphi  = \frac{ \theta _1+\theta _2}{\sqrt{2}}, \qquad \psi =  \frac{ \theta _1-\theta _2}{\sqrt{2}},
\]
can be written as $\phi(\theta,  (\varphi ,\psi ))=(\varphi +\theta ,\psi )$,
and we have the trivial principal $S^1$-bundle $p\colon S^1\times S^1\to S^1$, $(\varphi , \psi )\mapsto \psi$.
In addition, the quotient for the lifted $S^1$-action on $T^*M$  is $T^*S^1 \times \R$. We will consider coordinates $(\psi ,p_\psi ,y)$ where $y=(p_\varphi d\varphi )\cdot \varphi^{-1}$. Thus, the Hamiltonian is
\[
H_{(g,V)} (\psi , p_\psi ,y)=\frac{1}{2} ( p^2_\psi + y^2 ) +V(2\sqrt{\psi}).
\]
Let $e\in \R$ and the open set $U_e= \{\psi\in S^1\, \colon \, V(\sqrt{2}\psi)<e\}. $ Then 
\[
\mathbb{S}_{g_e}(T^*U_e\times \R)=\{(\psi , p_\psi, y)\in T^*U_e \, \colon \,  y^2 + p^2_\psi=2(e-V(\sqrt{2}\psi))\}
\]
and 
where $g_e$ is the metric on $U_e$ given by 
\[
g_e=2(e-V(\sqrt{2}\psi))g.
\]

The local expression of $X_{\kappa _{g_e}}$ is
\[
\begin{array}{rcl}
X_{\kappa _{g_e}}&=&
\displaystyle \frac{1}{2(e-V(\sqrt{2}\psi))} \Big ( \frac{1}{2} p_\psi \frac{\partial}{\partial \psi} + \frac{\partial V}{\partial \psi}
\frac{\partial}{\partial p_\psi}\Big )
\end{array}
\]
and the Jacobi structure on $T^*U_e\times \R$ is given by
\[
\begin{array}{rcl}
\Lambda_{T^*U_e\times \R}&=&\displaystyle \frac{\partial}{\partial \psi}\wedge \frac{\partial}{\partial p_\psi} +y \frac{\partial}{\partial y} \wedge
X_{\kappa _{g_e}}\\[8pt]
E_{T^*U_e\times \R}&=&
\displaystyle - \frac{1}{2(e-V(\sqrt{2}\psi))} \Big ( \frac{1}{2} p_\psi \frac{\partial}{\partial \psi} + \frac{\partial V}{\partial \psi}
\frac{\partial}{\partial p_\psi}\Big ).
\end{array}
\]
}
\end{example}

\subsection{Mechanical Hamiltonian systems with respect to linear semi-direct Poisson structures}

Let $G$ be a Lie group and  $\phi:G\times P\to P$ an action of $G$ on a manifold $P.$  If ${\mathfrak g}$ is the Lie algebra of $G$, then we consider the trivial vector bundle  $P\times {\mathfrak g\to P}$  whose space of sections may be identified with the set $C^\infty(P,{\mathfrak g})$ of ${\mathfrak g}$-valued functions on $P.$   The dual bundle, $P\times {\mathfrak g}^*$, admits a linear Poisson structure $\Pi_{P\times {\mathfrak g}^*}$, called the {\it semi-direct Poisson structure} \cite{KM,W3},  whose bracket of functions is given by
\begin{eqnarray*}
\{ F _1, F_2\}_{P\times {\mathfrak g}^*} (p,\alpha)&=&
- \alpha ( [(D_2 F_1) (p,\alpha ) , (D_2 F_2) (p,\alpha )   ]_\mathfrak{g} ) \\ && + ( D_2 F_1(p,\alpha) )_P(p) (F_2) -  ( D_2 F_2 (p,\alpha) )_P (p)(F_1),
\end{eqnarray*}
for $F_1,F_2\in C^\infty (P\times \mathfrak{g}^*)$ and  $(p,\alpha)\in P\times \mathfrak{g}^*$, where
$(D_2 F )(p,\alpha )=(d F (p,\cdot ) )(\alpha)\in T_{\alpha}^* \mathfrak{g}^*\cong \mathfrak{g}$ and, for every $\xi\in \mathfrak{g}$, $\xi _P\in\mathfrak{X}(P)$ is the infinitesimal generator of the action $\phi$ associated with $\xi$.
Additionally, suppose that there is a scalar product $
\langle \cdot,\cdot\rangle$ on ${\mathfrak g}$ and we 
denote by $\langle \cdot,\cdot\rangle^*$ the corresponding 
scalar product on ${\mathfrak g}^*$.  Let $H_{(\langle 
\cdot,\cdot\rangle, V)}\colon P\times {\mathfrak g}^*\to \R$ 
be the Hamiltonian function given by
\[
H_{(\langle  \cdot,\cdot\rangle, V)}(p,\alpha)=\frac{1}{2}\langle  \alpha,\alpha \rangle^* + V(p),\quad \mbox{ for } (p,\alpha )\in P\times \mathfrak{g}^*.
\]
Now, let $e\in \R$ a real number such that $U_e=\{p\in P\, :\,V(p)<e\}\not=\emptyset.$ Then 
\[
{\mathbb S}_{g_e}(P\times {\mathfrak g}^*)=H_{(\langle  \cdot,\cdot\rangle, V)}^{-1}(e)=(\kappa_{g_e})^{-1}\left (\frac{1}{2}\right )=\{(p,\alpha)\in P\times {\mathfrak g}^*\, :\, \langle \alpha,\alpha \rangle ^*=2(e-V(p))\},\]
where $\kappa_{g_e}$ is the kinetic energy induced by the Jacobi metric on $P\times {\mathfrak g}^*$
\[
g_e={2(e-V)}\langle \cdot,\cdot\rangle ^*.
\]
Moreover, the trajectories of the Hamilton equations for 
the Hamiltonian system $(P\times {\mathfrak g}^*, 
$ $H_{(\langle  \cdot,\cdot\rangle, V)})$ on ${\mathbb S}_{g_e}(
P\times {\mathfrak g}^*)$  are, up to reparametrization, 
 the trajectories of the Hamilton equations of the system $
 (P\times {\mathfrak g}^*, \kappa_{g_e})$  on ${\mathbb S}_{g_e}(P\times {\mathfrak g}^*).$ 

Now, let $\{ e_\alpha\}$ be a basis of $\mathfrak{g}$ such that 
\[
[e_\alpha ,e_\beta]_\g=c_{\alpha \beta}^\gamma e_\gamma ,\quad (e_\alpha )_P=-\rho ^i_\alpha \frac{\partial }{\partial q^i}.
\]
Then, the Jacobi structure on ${\mathbb S}_{g_e}(P\times {\mathfrak g}^*)$ is the restriction to ${\mathbb S}_{g_e}(P\times {\mathfrak g}^*)$  of the Jacobi structure $(\Lambda _{P\times {\mathfrak g}^*},E_{P\times {\mathfrak g}^*})$ on $P\times \g^*$ given by 
\[
\begin{array}{rcl}
\Lambda_{P\times\g^*}&=&\displaystyle{ (\rho ^i_\alpha -\frac{1}{2(e-V)}\rho^i_\gamma \langle\cdot ,\cdot\rangle ^{\gamma\beta} y_\alpha y_\beta )\displaystyle \frac{\partial
}{\partial
q^i}\wedge \displaystyle\frac{\partial }{\partial y_\alpha }}\\[10pt]
&&\displaystyle{ -\Big ( \frac{1}{2} c^\gamma_{\alpha\beta }y_\gamma + \frac{1}{2(e-V)} ( \rho ^i_\beta y_\alpha \frac{\partial V}{\partial q^i} - c_{\nu\beta}^\gamma \langle\cdot ,\cdot \rangle ^{\mu\nu} y_\alpha y_\gamma y_\mu ) \Big )\displaystyle \frac{\partial
}{\partial
y_\alpha }\wedge \displaystyle\frac{\partial }{\partial y_\beta },}\\[15pt]
E_{P\times\g^*}&=& \displaystyle{- \frac{1}{2(e-V)} \Big ( ( \rho^i_\alpha \langle\cdot ,\cdot\rangle ^{\alpha\beta}  y_\beta )\displaystyle \frac{\partial }{\partial q^i}-(\rho ^i_\beta  \frac{\partial V}{\partial q^i} - c_{\alpha\beta}^\gamma \langle\cdot ,\cdot \rangle ^{\alpha \mu} y_\gamma y_\mu )  \displaystyle\frac{\partial }{\partial y_\beta }\Big )}
\end{array}
\]
where $\langle\cdot,\cdot \rangle^{\alpha \beta}=\langle e^\alpha ,e^\beta \rangle ^*$.

So, the restriction to ${\mathbb S}_{g_e}(P\times {\mathfrak g}^*)$ of $X_{\kappa_{g_e}}$ is just 
$-(E_{P\times\g^*})_{|{\mathbb S}_{g_e}(P\times {\mathfrak g}^*)}$.
\begin{example}[{\bf The heavy-top}] {\rm Let $G$ be  the special orthogonal group $SO(3)$, $P$ be the sphere $S^2$  and the action $\phi: SO(3)\times S^2\to S^2$ given by $(A,q)\mapsto Aq$.  We take the basis $\{ e_1,e_2,e_3\}$ of $\so(3)$ such that
\[
[e_1,e_2]_{\so(3)}=e_3,\;\;\; [e_2,e_3]_{\so(3)}=e_1, \;\;\; [e_3,e_1]_{\so(3)}=e_2.
\]
Using this basis, we can identify $\so(3)$ with $\R^3$ and the infinitesimal vector fields 
of the action are described by $\xi _{S^2}(q)= \xi \times q$, for $\xi \in \R^3$ and $q\in 
S^2$, where $\times$  is the vector product in $\R^3$.

We consider the scalar product on $\mathfrak{so}(3)$ given by
\[
\langle x  , y \rangle =x \cdot I y,
\]
$I$ being the inertia tensor of the top.  For simplicity in the expressions, we will assume   
that $I$ is diagonal, that is,
\[
\langle (x^1,x^2,x^3),(y^1,y^2,y^3)\rangle =\sum_{i=1}^3 I_ix^i \, y^i,
\]
for $(x^1,x^2,x^3),(y^1,y^2,y^3)\in \mathbb{R}^3\cong 
\mathfrak{so}(3)$, where the positive numbers $I_1$, $I_2$, $I_3$ are the principal moments of inertia. The potential energy $V:S^2\to \R$  is 
defined by $V(q) = mgl(q\cdot a),$ for $q\in S^2,$ where 
$a$ is the unit vector from the fixed point to the center of 
mass and $m,$
$g$ and $l$ are constants. Thus, if $e\in \R$ and $U_e=\{ q\in S^2\, :\, q\cdot a <\displaystyle\frac{e}{mgl}\}$ 
\[
{\mathbb S}_{g_e}(S^2\times \mathfrak{so}(3)^*)=\{(q,\alpha)\in U_e \times \mathfrak{so}(3)^*\, : \, \sqrt{\alpha \cdot I^{-1}\alpha}=2(e-mgl(q\cdot a))\}.
\]
The trajectories of the Hamilton equations of $(S^2\times \mathfrak{so}(3)^*, H_{(\langle  \cdot,\cdot\rangle, V)})$ on ${\mathbb S}_{g_e}(S^2\times \mathfrak{so}(3)^*)$  are just
the trajectories of the Hamilton equations of the system $(S^2\times \mathfrak{so}(3)^*, \kappa_{g_e})$  on ${\mathbb S}_{g_e}(S^2\times \mathfrak{so}(3)^*)$, where $g_e$ is the metric on $S^2\times \mathfrak{so}(3)^*$
\[
g_e(x)(\alpha,\alpha')=2(e-mglx\cdot a)\alpha \cdot I^{-1}\alpha'.
\]
If $(y_1,y_2,y_3)$ are the coordinates on $\so(3)^*\cong \R^3$ with respect to the base $\{ e_1,e_2,e_3\}$ and $(q_1,q_2,q_3)$ are canonical coordinates on $\R^3$
then the linear Poisson structure $\Pi _{S^2\times \mathfrak{so}(3)^*}$ has the expression
\begin{eqnarray*}
\Pi _{S^2\times \mathfrak{so}(3)^*}&=&
 -q^3   \displaystyle\frac{\partial}{\partial q^1}\wedge \displaystyle\frac{\partial}{\partial y_2}
+ q^2
 \displaystyle\frac{\partial}{\partial q^1}\wedge \displaystyle\frac{\partial}{\partial y_3}
 + q^3
 \displaystyle\frac{\partial}{\partial q^2}\wedge \displaystyle\frac{\partial}{\partial y_1}
\\&& -q^1
\displaystyle\frac{\partial}{\partial q^2}\wedge \displaystyle\frac{\partial}{\partial y_3}
 -q^2
\displaystyle\frac{\partial}{\partial q^3}\wedge \displaystyle\frac{\partial}{\partial y_1}
+
 q^1
\displaystyle\frac{\partial}{\partial q^3}\wedge \displaystyle\frac{\partial}{\partial y_2}
\\&&  \displaystyle -\frac{1}{2} \Big ( y_3 
\frac{\partial}{\partial y_1}\wedge \displaystyle\frac{\partial}{\partial y_2} +y_2  
\frac{\partial}{\partial y_3}\wedge \displaystyle\frac{\partial}{\partial y_1} + y_1  
\frac{\partial}{\partial y_2}\wedge \displaystyle\frac{\partial}{\partial y_3}\Big ),
\end{eqnarray*}
the Hamiltonian vector field  $X_{\kappa_{g_e}}$ is given by
 \begin{eqnarray*}
X_{\kappa_{g_e}}&=&
- \frac{1}{2(e-V)}\displaystyle\Big [  \Big ( \frac{q^2 y_3}{I_3}-\frac{q^3 y_2}{I_2}  \Big)  \displaystyle\frac{\partial}{\partial y_1}
+
\Big ( \frac{q^3 y_1}{I_1} -\frac{q^1 y_3}{I_3} \Big )
 \displaystyle\frac{\partial}{\partial q^2}
+ 
\Big ( \frac{q^1 y_2}{I_2}  -\frac{q^2 y_1}{I_1}  \Big )
 \displaystyle\frac{\partial}{\partial q^3}
\\&&+ \Big ( q^3\frac{\partial  V}{\partial q^2}-q^2\frac{\partial  V}{\partial q^3}+\frac{y_2y_3}{I_2}-\frac{y_2y_3}{I_3} \Big )
 \displaystyle\frac{\partial}{\partial y_1}
 +\Big (q^1\frac{\partial  V}{\partial q^3} -q^3\frac{\partial  V}{\partial q^1}-\frac{y_1y_3}{I_1}+\frac{y_1y_3}{I_3} \Big )
 \displaystyle\frac{\partial}{\partial y_2}
 \\ && + \Big ( q^2\frac{\partial  V}{\partial q^1}-q^1\frac{\partial  V}{\partial q^2}+\frac{y_1y_2}{I_1}-\frac{y_1y_2}{I_2}\Big )
 \displaystyle\frac{\partial}{\partial y_3}\Big ]
\end{eqnarray*}
and the Jacobi structure is given by
\begin{eqnarray*}
\Lambda _{S^2\times \mathfrak{so}(3)^*}&=&\Pi_{S^2\times \mathfrak{so}(3)^*}+\Delta _{S^2\times \so(3)^*} \wedge  X_{\kappa_{g_e}},
\\
E_{S^2\times \mathfrak{so}(3)^*}&=&-X_{\kappa_{g_e}},
\end{eqnarray*}
where $\Delta _{S^2\times \so(3)^*} =y_\alpha \frac{\partial }{\partial y_\alpha}$ is the Liouville vector field.
}
\end{example}

\section{Conclusions and future work}
We have described the relation between the dynamics of mechanical Hamiltonian systems with respect to a linear Poisson structure on a vector bundle and the Jacobi-Reeb dynamics on the corresponding spherical bundle. In the particular case when the vector bundle is the cotangent bundle $T^*Q$ of a manifold $Q$ and the linear Poisson structure on $T^*Q$ is just the canonical symplectic structure, we recover the classical case.

\medskip
We remark that the classical result establishes an interesting connection between mechanical Hamiltonian flows on $T^*Q$ and geodesic flows associated with Riemannian metrics on $Q$ (or on open subsets $U_e$ of $Q$). Indeed, for a fixed energy $e$, the Reeb vector field of the contact structure on the spheric cotangent bundle $\S_{g_e}(T^*U_e)$, associated with the Jacobi metric $g_e$, is just the restriction to $\S_{g_e}(T^*U_e)$ of the geodesic flow on $T^*U_e$ for the Jacobi metric $g_e$. This fact has important dynamical consequences. For instance, when $Q$ is compact and the Riemannian curvature of $Q$ is negative (see Introduction). 

\medskip
In the more general case when we replace the cotangent bundle $T^*Q$ by an arbitrary vector bundle $\tau_{A^*} \colon A^*\to Q$ endowed with a linear Poisson structure, some aspects stay similar. More precisely, we have that $A$ is a Lie algebroid (see Remark \ref{rem:Liealg}) and if $g$ is a bundle metric on $A$, then one may introduce the Levi-Civita connection as an $A$-connection on $A$ (for the theory of such connections see, for instance, \cite{Ma} and the references therein) and the geodesics of $g$ as a particular class of admissible curves on $A$ (see \cite[Example 3.6]{CoLeMaMa}) or, equivalently, on $A^*$ (note that the metric $g$ induces a vector bundle isomorphism between $A$ and $A^*$). In fact, the geodesics are the integral curves of the Hamiltonian vector field $X_{\kappa _g}$ of the kinetic energy $\kappa _g\colon A^*\to \R$. $X_{\kappa _g}$ is the geodesic flow on $A^*$ of $g$. In addition, using the Levi-Civita connection, one can also introduce the curvature of the bundle metric $g$ and the sectional curvature of a 2-dimensional subspace of $A_q$, with $q\in Q$.

\medskip
The previous comments and Theorem \ref{estructura-S-estrella-e} imply, as in the classical case, that there also is a connection between mechanical Hamiltonian dynamics, with respect to linear Poisson structures on vector bundles, and geodesic flows on such spaces.

\medskip
So, it would be interesting to develop a research program on the geometry (geodesics and curvature) of bundle metrics on vector bundles endowed with linear Poisson structures (or, equivalently, bundle metrics on Lie algebroids). The results of this research could be applied in the description of mechanical Hamiltonian dynamics with respect to such Poisson structures. In fact, following the ideas in \cite{Anosov} for the classical case, some possible topics such as:
\begin{itemize}
\item Structural stability of the trajectories,
\item Periodic orbits,
\item Ergodicity of the system,
\item Killing tensors,
\end{itemize}
could be discussed. The case when the base space of the vector bundle is compact and the sectional curvature is negative should be remarkable. 

\medskip
Another interesting goal is to develop hybrid variational integrators for mechanical Hamiltonian systems with respect to linear Poisson structures 
based on the results in this paper and following the ideas in \cite{NBM}.

\medskip
Finally, we remark that there exists a relevant connection between the Jacobi-Maupertuis principle and the Schr\"odinger variational principle of wave mechanics (for a discussion on this topic, see \cite{GrKaNo}). So, it would be interesting to extend these ideas to the more general setting of Lie algebroids or, equivalently, vector bundles endowed with linear Poisson structures.

\medskip
{\bf Acknowledgments:} The authors acknowledge
financial support from the Spanish Ministry of Science and Innovation and European Union (Feder)  grant PGC2018-098265-B-C32. We are also grateful to the referees for their suggestions that helped us to improve this paper.

{\bf Data Availability} The data supporting the conclusions of this paper are included within the article

{\bf Conflicts of interest}  There are no conflict of interest in this article.

\end{document}